\documentclass[11pt]{article}
\usepackage{amsthm,amsfonts, amsbsy, amssymb,amsmath,graphicx}
\usepackage{graphics}
\usepackage[english]{babel}
\usepackage{graphicx}
\usepackage{subcaption}
\usepackage{lineno}
\usepackage{enumerate}
\usepackage{tikz}
\usepackage{tkz-graph}
\usepackage{tkz-berge}
\usepackage{hyperref}
\usepackage{color}
\usepackage{tikz}
\usetikzlibrary{arrows, automata}

\hypersetup{colorlinks=true}

\hypersetup{colorlinks=true, linkcolor=blue, citecolor=blue,urlcolor=blue}
\title{Independent coalition  in graphs : existance and characterization}
\author{{\small  Mohammad Reza Samadzadeh$^{a}$, Doost Ali Mojdeh$^{b}$\thanks{Corresponding author}}\\ {\small $^{a,b}$Department of
		Mathematics, Faculty of Mathematical Sciences}\\{\small University of Mazandaran, Babolsar, Iran}\\
	{\small $^a$m.samadzadeh02@umail.umz.ac.ir}\\ {\small $^b$damojdeh@umz.ac.ir} \\}
\date{}
\newtheorem{theorem}{Theorem}[section]
\newtheorem{corollary}[theorem]{Corollary}
\newtheorem{lemma}[theorem]{Lemma}

\newtheorem{observation}[theorem]{Observation}
\newtheorem{proposition}[theorem]{Proposition}

\theoremstyle{definition}
\newtheorem{definition}[theorem]{Definition}
\theoremstyle{remark}

\begin{document}
	
	\maketitle
	\begin{abstract}
			\noindent An independent coalition in a graph $G$ consists of two disjoint sets of vertices $V_1$ and $V_2$ neither of which is an independent dominating set but whose union $V_1 \cup V_2$ is an independent dominating set. An independent coalition partition, 
abbreviated, $ic$-partition, in a graph $G$ is a vertex partition $\pi= \lbrace V_1,V_2,\dots ,V_k \rbrace$ such that each set $V_i$ of $\pi$ either is a 
singleton dominating set,  or is not an independent dominating set but forms an independent coalition with another set $V_j \in \pi$. 
The maximum number of classes of an $ic$-partition of $G$ is  the independent coalition number of $G$, denoted by $IC(G)$. In this paper, we study the concept of $ic$-partition. In particular, we discuss the possibility of the existence of $ic$-partitions in graphs and introduce a family of graphs for which no $ic$-partition exists. We also determine the independent coalition number of some classes of graphs and investigate 	
graphs $G$ of order $n$ with $IC(G)\in\{1,2,3,4,n\}$ and the trees $T$ of order $n$ with $IC(T)=n-1$.
\end{abstract}


\section{Introduction}
\label{sec:introduction}
Let $G=(V,E)$ denote a  simple  graph of order $n$ with vertex set $V=V(G)$ and edge set $E=E(G)$. The {\em open
neighborhood} of a vertex $v\in V$ is the set $N(v)=\lbrace u \lvert \lbrace u,v\rbrace \in E \rbrace$, and its {\em closed
neighborhood} is the set $N[v]=N(v) \cup \lbrace v \rbrace$. Each vertex of $N(v)$ is called a {\em neighbor}
of $v$, and  the cardinality of $N(v)$ is called the {\em degree} of $v$, denoted by $deg(v)$ or $deg_G (v)$. A
vertex $v$ of degree $1$ is called a {\em pendant vertex} or {\em leaf}, and its neighbor is called a {\em support vertex}.  A vertex of degree $n-1$ is called a {\em full vertex} while a vertex of degree $0$ is called an {\em isolated vertex}. The {\em minimum} and  {\em maximum 
degree}  of $G$ are denoted by $\delta (G)$ and $\Delta (G)$, respectively.  For a set $S$ of vertices of $G$, the subgraph induced by $S$ is denoted by  $G[S]$.
For two  sets $X$ and $Y$ of vertices, let  $[X,Y]$ denote the set of edges between $X$ and $Y$.	
If every vertex of $X$ is adjacent to every vertex of  $Y$, we say that $[X,Y]$ is {\em full}, while
if there are no edges between them, we say that $[X,Y]$ is {\em empty}.
A subset $V_i \subseteq V$ is called a {\em singleton set} if $\lvert V_i \rvert =1$, and is called a {\em non-singleton set} if $\lvert V_i \rvert \geq 2$.
The {\em join} $G+H$  of two disjoint graphs $G$ and
$H$ is the graph obtained from the union of $G$ and $H$ by adding every possible
edge between  the vertices of $G$ and the vertices of $H$. 
We denote the family of paths, cycles, complete graphs and	 
stars of order $n$ by $P_n$, $C_n$, $K_n$ and $K_{1,n-1}$, respectively, and the complete $k$-partite graph with  partite sets of order   $n_1,n_2,\dots ,n_k$, by $K_{n_1,\dots ,n_k}$.
A double star
with respectively $p$ and $q$ leaves connected to each support vertex is denoted by $S_{p,q}$. The complete graph $K_3$ is called a {\em triangle}, and a graph is {\em triangle-free} if it has no $K_3$ as an induced subgraph.  The {\em girth} of a graph with a cycle is the length of its shortest cycle. For a graph $G$, the girth of $G$  is denoted by $g(G)$. For a graph $G$ of order $n$, let $\overline{G}$ denote the complement of $G$ with $V(\overline{G})=V(G)$ and $E(\overline{G})=E(K_n)-E(G)$ \cite{ref13}.

A set $S\subseteq V$  in a graph $G=(V,E)$ is called a {\em dominating set} if every vertex $v\in V$ is either an element of $S$ or is adjacent to an element of $S$. A set $S\subseteq V$ is called an {\em independent set} if its vertices are pairwise nonadjacent. The {\em vertex independence number}, denoted by $\alpha(G)$,
is the maximum cardinality of an independent set of $G$. An  {\em independent dominating set} in a graph $G$ is a set that is both independent an dominating. 

A partition of the vertices of $G$ into dominating sets (independent dominating sets) is called a domatic partition (idomatic partition).
The maximum number of classes of a domatic partition (idomatic partition) of $G$ is called the domatic number (idomatic number) of $G$, denoted by $d(G)$ ($id (G)$).
The concepts of domination and domatic partition and their variations have been studied widely in the literature. See, for example,  \cite{ref1,ref2,ref3,ref9,ref10,ref11,ref12}. 

The term {\em coalition} was introduced by   Haynes et al, \cite{ref6} and has been studied further in \cite{ref4,ref5,ref7,ref8}.
\begin{definition}
\cite{ref6}  A coalition in a graph $G$ consists of two disjoint sets of vertices $V_1,V_2 \subset V$, neither of which is a dominating set but whose union $V_1 \cup V_2$ is a dominating set.
We say that the sets $V_1$ and $V_2$ form a coalition, and are  {\em coalition partners}.
\end{definition}

\begin{definition}
\cite{ref6}
A {\em coalition partition}, henceforth called a {\em $c$-partition}, in a graph $G$ is a 
{\em vertex partition}	
$\pi =\lbrace V_1,V_2,\dots ,V_k \rbrace$   such that every set $V_i$ of $\pi$ is either a {\em singleton
	dominating set}, or 
is not a dominating set but
forms a coalition with another set $V_j$ in $\pi$. The {\em coalition number}
$C(G)$ equals the maximum order $k$ of a $c$-partition of $G$, and a $c$-partition of $G$ having
order $C(G)$ is called a {\em $C(G)$-partition}.
\end{definition}

Herein we will focus on coalitions involving independent dominating sets in graphs. In  other words, we will study the concepts of independent  coalition and independent  coalition partition which have been introduced in \cite{ref6} as an area for future research. We begin with the following definitions.

\begin{definition}\label{ic-def}

An {\em independent coalition} in a graph $G$ consists of  two disjoint sets of independent vertices $V_1$ and $V_2$, neither of which is an independent dominating set but whose union $V_1\cup V_2$ is an independent dominating set. We say the sets $V_1$ and $V_2$ form an independent coalition, and are  {\em independent coalition partners} (or {\em $ic$-partners}).

\end{definition}
\begin{definition}\label{icp-def}

An {\em independent coalition partition}, abbreviated {\em $ic$-partition}, in a graph $G$ is a vertex partition  $\pi=\{V_1,V_2,\dots ,V_k\}$  such that  every set $V_i$ of $\pi$ is either a singleton dominating set, or is not an independent dominating
set but forms an independent coalition with another set $V_j \in \pi$.  The {\em independent coalition number} $IC(G)$ equals the maximum number of classes of
an $ic$-partition of $G$, and an $ic$-partition of $G$ having order $IC(G)$ is called an {\em $IC(G)$-partition}.

\end{definition}

\begin{definition} \label{sp-def}
\cite{ref8}
Let $G$ be a graph of order $n$ with vertex set $V=\lbrace v_1,v_2, \dots ,v_n \rbrace$. The {\em singleton partition}, denoted $\pi_1$, of $G$ is the partition of $V$ into $n$ singleton sets, that is, $\pi_1 =\lbrace \lbrace v_1\rbrace ,\lbrace v_2 \rbrace ,\dots ,\lbrace v_n \rbrace \rbrace$.
\end{definition}

This paper is organized as follows. Section 1 is devoted to terminology and definitions. We  discuss the possibility of the existence of $ic$-partitions in graphs and  derive  some bounds on independent coalition number  in Section 2. In Section 3, we determine the independent coalition number of some classes of graphs. The graphs $G$ with $IC(G)\in \lbrace 1,2,3,4\rbrace$   are investigated in 	
Section 4. In Section 5, we characterize triangle-free graphs $G$ with $IC(G)=n$ and trees $T$ with $IC(T)=n-1$. Finally, we end the paper with some research problems.

\section{Independent coalition partition: existence and bound}
\label{sec:second}

This section is divided into two subsections. In the first subsection, we show  that not all graphs admit an $ic$-partition, and in the second subsection, we present some bounds on $IC(G)$ whenever the  graph $G$ admits an $ic$-partition.

\subsection{Existence}
In the following definition, we construct graphs with arbitrarily large order for which no $ic$-partition exists.
\begin{definition}
Let $\mathcal{B}$ be the set of all  graphs obtained  from the complete
graph $K_n,$ ($n\ge 4$)  with the vertices $v_i$, $(1\leq i \leq n)$,  and two additional vertices $v_{n+1}, v_{n+2}$ such that $v_{n+1}$ and $v_{n+2}$ are adjacent to $v_n$, and  $v_{n+1}$ is adjacent to $v_{n-1}$. 
Figure \ref{pic1} illustrates such a graph for $n=4$.
\end{definition}
\begin{figure}[!htbp]
\centering
\begin{tikzpicture}[scale=.3, transform shape]
	\node [draw, shape=circle,fill=black] (v1) at  (0,5) {};
	\node [draw, shape=circle,fill=black] (v2) at  (0,0) {};
	\node [draw, shape=circle,fill=black] (v3) at  (5,0) {};
	\node [draw, shape=circle,fill=black] (v4) at  (5,5) {};
	\node [draw, shape=circle,fill=black] (v5) at  (10,0) {};
	\node [draw, shape=circle,fill=black] (v6) at  (15,0) {};
	\node [scale=3] at (0,5.7) {$v_1$};
	\node [scale=3] at (0,-0.7) {$v_2$};
	\node [scale=3] at (5,-0.7) {$v_3$};
	\node [scale=3] at (5,5.7) {$v_4$};
	\node [scale=3] at (10,-0.7) {$v_5$};
	\node [scale=3] at (15,-0.7) {$v_6$};
	\draw(v1)--(v2);
	\draw(v2)--(v3);
	\draw(v3)--(v4);
	\draw(v4)--(v1);
	\draw(v1)--(v3);
	\draw(v2)--(v4);
	\draw(v3)--(v5);
	\draw(v4)--(v5);
	\draw(v4)--(v6);
	
\end{tikzpicture}
\caption{The graph $G$ in $\mathcal{B}$ for $n=4$}\label{pic1}
\end{figure} 
\begin{proposition}\label{claim-1}
Let $G$ be a graph. If $G\in \mathcal{B}$, then $G$ has no $ic$-partition.
\end{proposition} 

\begin{proof}

Suppose, to the contrary, that $G$ has an $ic$-partition $\pi$. The vertices  $v_1,v_2,\dots ,v_{n-1}$ are pairwise adjacent, so they must be in different classes. Further, $v_{n}$ is a full vertex, so it must be in a singleton class. Since $v_{n-1}$ is adjacent to all vertices except $v_{n+2}$, and $\lbrace v_{n-1} ,v_{n+2}\rbrace$ dominates $G$, it follows that  $\lbrace v_{n-1}\rbrace \in \pi$. Further, since $\lbrace v_{n-1}\rbrace$ can only form an independent coalition with  $\lbrace v_{n+2}\rbrace$, it follows that $\lbrace v_{n+2}\rbrace \in \pi$. If  $\{v_{n+1}\} \in \pi$, then $\pi$ is a singleton partition. In this case,  $\{v_{n+1}\}$ has no $ic$-partner, a contradiction. Hence, $\{v_{n+1}\} \notin \pi$. It follows that
$\pi$ consists of a non-singleton set $\lbrace v_{n+1},v_i\rbrace$ such that $v_i\in \lbrace v_1,v_2,\dots ,v_{n-2}\rbrace$, and $n$ singleton sets. Assume, without loss of generality, that $\{v_{n+1},v_1\} \in \pi$. Now  for each $2\leq i \leq n-2$, the set $\lbrace v_i\rbrace$ has no $ic$-partner, a contradiction.

\end{proof}
\subsection{Bounds}

Definition \ref{icp-def} implies that an $ic$-partition of a graph $G$ is also a $c$-partition. Further, we note that an $ic$-partition of $G$ is  a proper coloring as well. Hence, we have the following two sharp bounds on $IC(G)$. To see the sharpness of them, consider the complete graph $K_n$. 
\begin{observation} \label{obs1}
Let $G$ be a graph. If $G$ has an $ic$-partition, then  $IC(G) \leq C(G)$. Furthermore, this bound is sharp.
\end{observation}
\begin{observation}\label{obs2}
Let $G$ be a graph. If $G$ has an $ic$-partition, then  $IC(G) \geq \chi (G)$. Furthermore, this bound is sharp.  
\end{observation} 
Given a connected graph $G$ and an $ic$-partition $\pi$ of it,  the following theorem  shows that each set in $\pi$ admits at most $\Delta (G)$ $ic$-partners.

\begin{theorem}\label{the-del}
Let $G$ be a connected graph with maximum degree $\Delta (G)$, and let $\pi$ be an  $ic$-partition of $G$. If $X \in \pi$, then $X$ is in at most $\Delta (G)$ independent coalitions. Furthermore, this bound is sharp.
\end{theorem}
\begin{proof}
Let $\pi$ be an $ic$-partition of $G$, and let $X$ be a set in $\pi$. If $X$ is a dominating set, then  it has no $ic$-partner.
Hence, we may assume that $X$  does not dominate $G$. Let $x$ be a
vertex that is not dominated by $X$. Now every $ic$-partner of $X$ must dominate $x$, that is, it must contain a
vertex in $N[x]$. Hence, there are at most $\lvert N[x] \rvert \leq \Delta (G)+1$ sets in $\pi$ that can form an independent  coalition with $X$. Now we show that $X$ cannot form an independent coalition with  $\Delta (G)+1$ sets. Suppose, to the contrary, that  $X$ has $\Delta (G)+1$ $ic$-partners (name $V_1, V_2, \dots ,V_{\Delta+1}$). Consequently,  $[X,V_i]$ is empty for each $1\leq i \leq \Delta (G)+1$. Let
$U=\bigcup _{i=1}^{\Delta (G)+1} V_i$, and $G^\prime =G[U]$. Consider an arbitrary vertex $v \in U$ (say $v \in V_i$) and an arbitrary  set $V_j$ such that  $1 \leq j \leq \Delta(G)+1$ and $j \neq i$. Since $X \cup V_j$ dominates $G$ and $[X,V_i]$ is empty, it follows that $v$ has a neighbor in $V_j$. Choosing $V_j$ arbitrarily, we conclude that $deg_{G^\prime} (v) \geq \Delta(G)$, and so $deg_{G^\prime}(v)=\Delta(G)$. Hence, for each $v\in U$, we have $deg_{G^\prime}(v)=\Delta(G)$.  Now  since $G$ is connected, there is a path $P=(v_0,v_1,\dots ,v_k)$ connecting $U$  to $X$ such that $v_0 \in U$ and $v_k \in X$. Note that $[U,X]$ is empty, and so  $V(P) \setminus (U \cup X) \neq \emptyset$. Let $i$ be the smallest index for which $v_i \notin U \cup X$. It follows that $v_{i-1} \in U$, and so $deg_{G^\prime} (v_{i-1}) = \Delta(G)$. Thus, we have $deg_{G} (v_{i-1}) \geq deg_{G^\prime} (v_{i-1})+1 = \Delta(G) +1$, a contradiction.  

To prove the sharpness, let  $G$ be the graph  that is obtained from the complete graph $K_n$ with vertices
$v_i$, ($1\le i\le n$), and a path  $P_2=(a,b)$, where $b$ is adjacent to $v_1$. Let $A=\lbrace a\rbrace$, $B=\lbrace b\rbrace$ and $V_i=\lbrace v_i\rbrace$, for $1\leq i\leq n$. One can observe that $\Delta (G)=n$ and  that  the singleton partition $\pi_1 =\lbrace V_1,V_2,\dots ,V_n ,A ,B \rbrace$ is an $ic$-partition  of $G$ such that $A$ forms an independent coalition with  $V_i$, for each $1\leq i \leq n$. This completes the proof.
\end{proof}

Note that the bound presented in Theorem \ref{the-del} does not hold for disconnected graphs. As a counterexample, consider the graph $G=K_2 \cup K_2$ and the singleton partition $\pi_1$ of it. On can verify that $\pi_1$ is an $ic$-partition of $G$ such that each set in $\pi_1$ has two $ic$-partners, while $\Delta (G)=1$.  

The next bound relates independent coalition number of a graph to its idomatic number. As we will see in the proof of Theorem \ref{the-doma}, any graph admitting an idomatic partition has  an  $ic$-partition. However, the converse is not necessarily true. For example, the singleton partition of the cycle $C_5$ is  an $ic$-partition of it, while $C_5$  has no idomatic partition.
Or the cycle  $C_7$ has the $ic$-partition $\pi=\lbrace \lbrace v_1,v_5\rbrace$, $\lbrace v_2,v_4\rbrace, \lbrace v_3\rbrace, \lbrace v_6\rbrace, \lbrace v_7\rbrace \rbrace$, while it has no idomatic partition.

\begin{theorem}\label{the-doma}
Let $G$ be a connected graph, and let
$r\geq 0$ be the number of full vertices of $G$. If $G$ admits an idomatic partition, then $IC(G) \geq 2id(G)-r$.
\end{theorem}
\begin{proof}
Let  $F=\lbrace v_1,v_2,\dots ,v_r\rbrace$ be the set of full vertices of $G$, and let $\pi =\lbrace V_1 ,V_2 ,\dots ,V_{id(G)} \rbrace$ be an idomatic partition of  $G$ of order $id(G)$. Note  that each full vertex must be in a singleton set of $\pi$. Without loss of generality, assume that $v_i \in V_i$, for each $1\leq i\leq r$. It follows that  for each $r+1\leq i\leq k$, we have $\lvert V_i\rvert \geq 2$. Now for each $r+1\leq i\leq k$, we partition $V_i$ into two  nonempty subsets $V_{i,1}$ and $V_{i,2}$. Note that no proper subset of $V_i$ is a dominating set. Thus, neither $V_{i,1}$ nor $V_{i,2}$ is an independent dominating set, and so $V_{i,1}$ and $V_{i,2}$ are $ic$-partners. It follows that the partition $\pi^\prime =\lbrace V_1,V_2,\dots ,V_r, V_{r+1,1} ,V_{r+1,2}, V_{r+2,1} ,V_{r+2,2} , \dots ,V_{id(G),1} ,V_{id(G),2}\rbrace$ is an $ic$-partition of $G$ of order $2id(G)-r$. Hence, $IC(G) \geq 2id(G)-r$.
\end{proof}
\section{Independent coalition number for some classes of graphs}
\label{sec:theorems}
Let us  begin this section  with some routine results.

\begin{observation} \label{obs-comp}
For $n\geq 1$, we have $IC(K_n)=n$.
\end{observation} 
\begin{observation} \label{obs-star}
For $n\geq 3$, we have $IC(K_{1,n-1})=3$.
\end{observation}
\begin{observation} \label{obs-dstar}
For $p,q\geq 1$, we have $IC(S_{p,q})=4$.
\end{observation}
For complete multipartite graphs, the following result is obtained.
\begin{proposition}\label{prop5}
Let $G=K_{n_1,n_2,\dots ,n_k}$ be a complete $k$-partite graph with $m\geq 0$ full vertices ($m$ partite sets of cardinality $1$). Then $IC(G)=2k-m$.
\end{proposition}
\begin{proof}

Let $\pi =\lbrace V_1,V_2,\dots ,V_k\rbrace$ be the partition of $G$ into its partite sets. 
Assume, without loss of generality, that the sets $V_i$, for $1\leq i\leq m$, are those containing full vertices. Now for each  $m+1 \leq i\leq k$, we partition $V_i$ into two sets  $V_{i,1}$ and $V_{i,2}$. Observe that $V_{i,1}$ and $V_{i,2}$ are $ic$-partners, and so the partition \\ $\pi^\prime=\lbrace \lbrace V_1\rbrace ,\lbrace V_2\rbrace ,\dots ,\lbrace V_m\rbrace,\lbrace V_{m+1,1},V_{m+1,2}\rbrace ,\lbrace V_{m+2,1},V_{m+2,2}\rbrace , \dots ,\lbrace V_{k,1},V_{k,2}\rbrace \rbrace$ is an $ic$-partition of $G$ of order $2k-m$. Thus, $IC(G)\geq 2k-m$. Now  let $\pi^{\prime\prime}$ be an $ic$-partition of $G$. We note that $\pi^{\prime\prime}$ has the following properties:
\begin{itemize}
	\item
	For any set $S \in \pi^{\prime \prime}$,
	all vertices in $S$ are  in the same partite set of $G$.
	\item
	For any set $V_i \in \pi$, the vertices in $V_i$ are  in at most two sets of $\pi^{\prime\prime}$.
\end{itemize}

Hence, we have $IC(G)\leq 2k-m$, and so $IC(G)=2k-m$.
\end{proof}

Next we determine the independent coalition number of all paths and cycles.
\begin{lemma} \cite{ref6} \label{lempath}
For any path $P_n$, $C(P_n)\leq 6$.
\end{lemma}
\begin{theorem}\label{the-path}
For the path $P_n$,
\[
IC(P_n) = \begin{cases}
	n & \text{if } n \leq 4; \\
	4 & \text{if } n=5; \\
	5 & \text{if } n=6,7,8,9; \\
	6 & \text{if } n\geq 10.
	
\end{cases}
\]
\end{theorem} 
\begin{proof}
It is clear that for $1\leq n\leq 4$, we have $IC(P_n)=n$. Now
let $n=5$. Consider the path $P_5$ with $V(P_5)=\lbrace v_1,v_2,v_3,v_4,v_5\rbrace$. 	It is easily seen
that $IC(P_5)\neq 5$. Thus, $IC(P_5)\leq 4$. The partition $\lbrace \lbrace v_1,v_3\rbrace ,\lbrace v_2\rbrace,\lbrace v_4\rbrace,\lbrace v_5\rbrace \rbrace$ is an $ic$-partition of $P_5$, so $IC(P_5)=4$.
Now  assume  $n=6$. Consider the path $P_6$ with $V(P_6)=\lbrace v_1,v_2,v_3,v_4,v_5,v_6\rbrace$. It is clear that $IC(P_6)\neq 6$. The partition $\lbrace \lbrace v_1,v_6\rbrace ,\lbrace v_2\rbrace,\lbrace v_3\rbrace,\lbrace v_4\rbrace,\lbrace v_5\rbrace \rbrace$ is an $ic$-partition of $P_6$, so $IC(P_6)=5$.
Next assume  $n=7$. Consider the path $P_7$ with $V(P_7)=\lbrace v_1,v_2,v_3,v_4,v_5,v_6,v_7\rbrace$. By Lemma \ref{lempath} and Observation \ref{obs1}, we have $IC(P_n)\leq 6$. Now
we show that $IC(P_7)\neq 6$. Suppose that $IC(P_7)= 6$. Let $\pi$ be an $IC(P_7)$-partition. We note that $\pi$ consists of a set (name $A$) of cardinality $2$ and five singleton sets. Since $\gamma_i (P_7)=3$, each singleton set must be an $ic$-partner of $A$.  On the other hand,
Theorem \ref{the-del} implies that $A$ has at most two $ic$-partners, a contradiction.
The partition $\lbrace \lbrace v_1,v_6\rbrace ,\lbrace v_2,v_7\rbrace,\lbrace v_3\rbrace,\lbrace v_4\rbrace ,\lbrace v_5\rbrace \rbrace$ is an $ic$-partition of $P_7$. Therefore, $IC(P_7)=5$. Next we assume $n=8$. Consider the path $P_8$ with $V(P_8)=\lbrace v_1,v_2,v_3,v_4,v_5,v_6,v_7,v_8\rbrace$.
By Lemma \ref{lempath} and Observation \ref{obs1}, we have $IC(P_n)\leq 6$. Now we show that $IC(P_8)\neq 6$. Suppose that $IC(P_8)=6$. Let $\pi$ be an $IC(P_8)$-partition. We consider two cases.

\textbf{Case 1.} $\pi$ consists of a set (name $A$) of cardinality $3$ and five singleton sets. Since $\gamma_i (P_8)=3$, each singleton set must be an $ic$-partner of $A$. On the other hand, Theorem \ref{the-del} implies that $A$ has at most two $ic$-partners, a contradiction. 

\textbf{Case 2.} $\pi$ consists of two sets  of cardinality $2$ and four singleton sets.  Since $\gamma_i (P_8)=3$, each singleton set must be an $ic$-partner of a set of cardinality $2$. Therefore, using Theorem \ref{the-del}, we deduce that for any two $ic$-partners $C$ and $D$, it holds that $\lvert C \cup D\rvert =3$. On the other hand, $v_3$ and $v_6$ are not present in any independent dominating set of cardinality $3$, a contradiction.\\
The partition $\lbrace \lbrace  v_1,v_3, v_6 \rbrace, \lbrace  v_2,v_7 \rbrace, \lbrace v_8\rbrace,
\lbrace v_4\rbrace, \lbrace v_5\rbrace\}$ is an $ic$-partition of $P_8$. Therefore, $IC(P_8)=5$.

Now let $n=9$. Consider the path $P_9$ with $V(P_9)=\lbrace v_1,v_2,v_3,v_4,v_5,v_6,v_7,v_8,v_9\rbrace$. By Lemma \ref{lempath} and Observation \ref{obs1}, we have $IC(P_n)\leq 6$. Now we show that $IC(P_9)\neq 6$. Suppose that $IC(P_9)=6$. Let $\pi$ be an $IC(P_9)$-partition. There exist three cases. 

\textbf{Case 1.} $\pi$ consists of a set (name $A$) of cardinality $4$ and five singleton sets. Since $\gamma_i (P_9)=3$, each singleton set must be an $ic$-partner of $A$. On the other hand,
by Theorem \ref{the-del}, $A$ has at most two $ic$-partners, a contradiction. 

\textbf{Case 2.} $\pi$ consists of a set (name $A$) of cardinality $3$, a set (name $B$) of cardinality $2$  and four singleton sets. Since $\gamma_i (P_9)=3$, no two singleton sets in $\pi$ are $ic$-partners. Furthermore, by Theorem \ref{the-del}, $A$ has at most two $ic$-partners, so at least two singleton sets of $\pi$ must be $ic$-partners of $B$, which is impossible, as $P_9$ has a unique independent dominating set of cardinality $3$.  

\textbf{Case 3.} $\pi$ consists of three sets of cardinality $2$, and three singleton sets. We note that each singleton set in $\pi$ must be an $ic$-partner of a set of cardinality $2$, which is impossible, as $P_9$ has a unique independent dominating set of cardinality $3$.

The partition $\lbrace \lbrace v_1,v_3,v_5\rbrace ,\lbrace v_2,v_4,v_9\rbrace,\lbrace v_6\rbrace,\lbrace v_7\rbrace,\lbrace v_8\rbrace \rbrace$ is an $ic$-partition of $P_9$. Therefore, $IC(P_9)=5$.

Finally, Let $n\geq 10$. Consider the path $P_n$ with $V(P_n)=\lbrace v_1,v_2, \dots ,v_n \rbrace$. By Lemma \ref{lempath} and Observation \ref{obs1}, we have $IC(P_n)\leq 6$. Now we consider the sets $V_1=\{v_1,v_6\}\cup \{v_{2n-1}: n\ge 5\}$, $V_2=\{v_2,v_5\}\cup \{v_{2n}: n\ge 5\}$, $V_3=\{v_3\}$, $V_4=\{v_4\}$, $V_5=\{v_7\}$, $V_6=\{v_8\}$. Then  $\pi=\{V_1,V_2,V_3, V_4,V_5,V_6\}$ is an $ic$-partition of $P_{n}$,  where  $V_3$ and $V_4$ are $ic$-partners of $V_1$, and  $V_5$ and $V_6$ are $ic$-partners of $V_2$. So the proof is complete.
\end{proof} 
\begin{lemma} \cite{ref6} \label{l3}
For any cycle $C_n$, $C(C_n)\leq 6$.
\end{lemma}
Lemma \ref{l3} and Observation \ref{obs1} imply the following result.
\begin{lemma} \label{l4}
For any cycle $C_n$, $IC(C_n)\leq 6$.
\end{lemma}
\begin{lemma} \label{cycle2}
For any cycle $C_n$ with $n\geq 8$ and $n \equiv 0\ (\textrm{mod}\ 2)$, it holds that  $IC(C_n)=6$.
\end{lemma}
\begin{proof} 
Let $V(C_n)=\lbrace v_1,v_2,\dots ,v_n\rbrace$. Consider the sets $V_1=\lbrace v_1,v_6\rbrace \cup \lbrace v_{2n-1} : n\geq 5\rbrace$, $V_2=\lbrace v_2,v_5\rbrace \cup \lbrace v_{2n} : n\geq 5\rbrace$, $V_3=\lbrace v_3\rbrace$, $V_4=\lbrace v_4\rbrace$, $V_5=\lbrace v_7\rbrace$, $V_6=\lbrace v_8\rbrace$. Then $\pi=\{V_1,V_2,V_3, V_4,V_5,V_6\}$ is an $ic$-partition of $C_n$, for $n\geq 8$,  where
$V_3$ and $V_4$ are $ic$-partners of $V_1$, and  $V_5$ and $V_6$ are $ic$-partners of $V_2$.  Hence, by Lemma \ref{l4} and Observation \ref{obs1}, we have $IC(C_n)=6$. 
\end{proof}
\begin{lemma} \label{cycle3}
For any cycle $C_n$ with $n\geq 8$ and $n \equiv 0\ (\textrm{mod}\ 3)$, it holds that $IC(C_n)=6$.
\end{lemma}
\begin{proof}
Let $V(C_n)=\lbrace v_1,v_2,\dots ,v_{3k} \rbrace$. Consider the sets $V_1 =\lbrace v_{3i+1}\rbrace$, $V_2 =\lbrace v_{3i+2} \rbrace$ and $V_3 =\lbrace v_{3i+3} \rbrace$, for $0\leq i\leq k-1$. Now for each $1\leq i\leq 3$, we  partition $V_i$ into two nonempty sets $V_{i,1}$ and $V_{i,2}$. Observe  that $V_{i,1}$ and $V_{i,2}$ are $ic$-partners.  Hence, by Lemma \ref{l4} and Observation \ref{obs1}, we have $IC(C_n)=6$.
\end{proof}
\begin{lemma} \label{cycle5}
For any cycle $C_n$ with $n\geq 8$ and $n \equiv 5\ (\textrm{mod}\ 6)$, it holds that $IC(C_n)=6$.
\end{lemma}
\begin{proof}
Assume $n=6k-1$, ($k\geq 2$).  Let $V(C_n)=\lbrace v_1,v_2,\dots ,v_n\rbrace$. Consider the sets
\[
A=\bigcup_{i=0}^{k-1} \lbrace v_{3i+1} \rbrace, \  A_1=\bigcup_{i=k}^{2k-1} \lbrace v_{3i+1} \rbrace, \  
A_2=\bigcup_{i=k}^{2k-1} \lbrace v_{3i} \rbrace,  
\]
\[
B=\bigcup_{i=k}^{2k} \lbrace v_{3i-1} \rbrace, \ 
B_1=\bigcup_{i=1}^{k-1} \lbrace v_{3i-1} \rbrace, \ 
B_2=\bigcup_{i=1}^{k-1} \lbrace v_{3i} \rbrace.		
\]

Let $\pi=\lbrace A,A_1,A_2,B,B_1,B_2\rbrace$. One can observe that $\pi$ is an $ic$-partition of $C_n$,  where $A_1$ and $A_2$ are $ic$-partners of $A$, and   $B_1$ and $B_2$ are $ic$-partners of $B$. Now using Lemma \ref{l4} and Observation \ref{obs1}, we have $IC(C_n)=6$.
\end{proof}
\begin{lemma} \label{cycle1}
For any cycle $C_n$ with $n\geq 8$ and $n \equiv 1\ (\textrm{mod}\ 6)$, it holds that $IC(C_n)=6$.
\end{lemma}
\begin{proof}
Assume  $n=6k+1$, ($k\geq 2$). Let $V(C_n)=\lbrace v_1,v_2,\dots ,v_n\rbrace$. Consider the sets

\[
A=\left( \bigcup_{i=0}^{k} \lbrace v_{3i+1} \rbrace \right) \cup \lbrace v_{3k+3}\rbrace, \ 
A_1=\bigcup_{i=k+2}^{2k} \lbrace v_{3i} \rbrace, \ 
A_2=\bigcup_{i=k+2}^{2k} \lbrace v_{3i-1} \rbrace,	 
\]
\[
B=\left( \bigcup_{i=k+1}^{2k} \lbrace v_{3i+1} \rbrace \right) \cup \lbrace v_{3k+2}\rbrace, \ 
B_1=\bigcup_{i=1}^{k} \lbrace v_{3i-1} \rbrace, \ 
B_2=\bigcup_{i=1}^{k} \lbrace v_{3i} \rbrace.
\] 
Let $\pi=\lbrace A,A_1,A_2,B,B_1,B_2\rbrace$. One can observe that $\pi$ is an  $ic$-partition of $C_n$,  where  $A_1$ and $A_2$ are $ic$-partners of $A$, and $B_1$ and $B_2$ are $ic$-partners of $B$. Now using Lemma \ref{l4} and Observation \ref{obs1}, we have $IC(C_n)=6$.
\end{proof}
\begin{theorem} \label{th-cycle}
For the cycle $C_n$,
\[
IC(C_n) = \begin{cases}
	n & \text{if } n \leq 6; \\
	5 & \text{if } n=7; \\
	6 & \text{if } n\geq 8.
	
\end{cases}
\]
\end{theorem}
\begin{proof}
If $1\leq n\leq 6$, then it is easy to check that $IC(C_n)=n$. Now assume $n=7$. Consider the cycle  $C_7$  with $V(C_7)=\lbrace v_1,v_2,v_3,v_4,v_5,v_6,v_7 \rbrace$. First we show that $IC(C_7)\neq 6$. Suppose, to the contrary, that $IC(C_7)=6$. Let $\pi$ be an $IC(C_7)$-partition. We note that $\pi$ consists of  five singleton sets and a set of cardinality $2$ (name $A$). By Theorem \ref{the-del}, $A$ has at most two $ic$-partners. Hence, $\pi$ contains two singleton sets that are $ic$-partners, which contradicts the fact that  $\gamma_i (C_7)=3$. The partition  $\lbrace \lbrace v_1,v_3\rbrace ,\lbrace v_5\rbrace ,\lbrace v_6\rbrace ,\lbrace v_4 ,v_7\rbrace ,\lbrace v_2\rbrace \rbrace$ is an $ic$-partition of $C_7$, so $IC(C_7)=5$.
Furthermore, by
Lemmas \ref{cycle2}, \ref{cycle3}, \ref{cycle5} and \ref{cycle1} we have $IC(C_n)=6$, for $n\geq 8$.
\end{proof}
\section{Graphs  with small independent coalition number}
\label{sec:morethms}
In this section we investigate graphs $G$ with $IC(G) \in \lbrace 1,2,3, 4 \rbrace$. We will make use of the following two lemmas.
\begin{lemma} \label{lemfull}
	Let $G$ be a graph of order $n$ containing 
	$r\geq 1$ full vertices, and let $F=\lbrace v_1,v_2,\dots ,v_r\rbrace$ be the set of  full vertices of $G$. Then $IC(G)=k$, if and only if $IC(G[V\setminus F])=k-r$, where $r <k \leq n$.
\end{lemma}
\begin{proof}
	Assume first that $IC(G[V\setminus F])=k-r$. Let $\pi =\lbrace V_1,V_2,\dots ,V_{k-r} \rbrace$ be an  $IC(G[V\setminus F])$-partition. Now the partition $\pi^\prime =\lbrace V_1,V_2,\dots ,V_{k-r}, \lbrace v_1\rbrace,\lbrace v_2\rbrace ,\dots ,\lbrace v_r \rbrace \rbrace$, is an $ic$-partition of $G$, so $IC(G)\geq k$. Now we prove that $IC(G)= k$. Suppose, to the contrary, that  $IC(G) > k$. Let $\pi$ be an $IC(G)$-partition. Now the partition $\pi^\prime =\pi \setminus \lbrace \lbrace v_1 \rbrace,\lbrace v_2\rbrace ,\dots ,\lbrace v_r \rbrace \rbrace$ is an $ic$-partition of $G[V\setminus F]$ such that $\lvert \pi^\prime \rvert > k-r$, a contradiction. Hence, $IC(G)= k$. Conversely, assume that $IC(G)=k$. Let $\pi$ be an $IC(G)$-partition. Now the partition  $\pi^\prime =\pi \setminus \lbrace \lbrace v_1\rbrace ,\lbrace v_2\rbrace , \dots ,\lbrace v_r\rbrace \rbrace$ is an $ic$-partition of $G[V\setminus F]$, so $IC(G[V\setminus F])\geq k-r$. Now we prove that $IC(G[V\setminus F])=k-r$. Suppose, to the contrary, that  $IC(G[V\setminus F]) > k-r$. Let $\pi$ be an $IC(G[V\setminus F])$-partition. Now the partition $\pi^\prime =\pi \cup \lbrace \lbrace v_1\rbrace ,\lbrace v_2\rbrace ,\dots ,\lbrace v_r\rbrace \rbrace $ is an $ic$-partition of $G$ such that $\lvert \pi^\prime \rvert > k$, a contradiction. Hence, $IC(G[V\setminus F])=k-r$. 
\end{proof}
\begin{lemma} \label{lemiso}
	Let $G$ be a graph containing a nonempty set  of isolated vertices $I$. If $IC(G)\geq 3$, then for any  $IC(G)$-partition $\pi$, there is a set $V_r\in \pi$ such that $V_r=I$.
\end{lemma}
\begin{proof}
	First we show that all vertices in $I$ are in the same set of $\pi$. Suppose, to the contrary, that there are sets $V_i \in \pi$ and $V_j \in \pi$
	such that both $V_i$ and $V_j$ contain isolated vertices. Let $V_k \in \pi$ be an arbitrary set in $\pi$ such that $V_k \notin \lbrace V_i,V_j\rbrace$. (Since $IC(G)\geq3$, such a set exists). Then $V_k$ has no $ic$-partner, a contradiction. Now let $V_r$ be the set in $\pi$ containing isolated vertices. Further, let $v$ be an arbitrary vertex in $V_r$, and let $u\in V(G)$ be an arbitrary vertex such that $u\neq v$.  If $u\in V_r$, then $u$ is not adjacent to  $v$. Otherwise, the set in $\pi$  containing $u$ is an $ic$-partner of $V_r$, which again implies that $u$ is not adjacent to $v$. Hence, we have $deg(v)=0$. Choosing $v$ arbitrarily, we conclude that $V_r=I$.
\end{proof}
\begin{proposition}\label{prop2}
	Let $G$ be a graph of order $n$. Then
	\begin{enumerate}
		\item
		$IC(G)=1$ if and only if $G \simeq K_1$.
		\item
		$IC(G)=2$ if and only if $G\simeq K_2$ or $G\simeq \overline{K}_n$, for some $n\geq 2$.		
	\end{enumerate}
\end{proposition}
\begin{proof}
	$(1)$ It is clear that 	$IC(G)=1$ if and only if $G \simeq K_1$. \\
	$(2)$ If $G\simeq K_2$, then we clearly  have $IC(G)=2$. Now
	assume  $G\simeq \overline{K}_n$, for some $n\geq 2$. Let $\pi$ be an $ic$-partition of $G$. Note that no more than two sets in $\pi$  contain isolated vertices,  for otherwise, no two sets in $\pi$ are $ic$-partners. Thus, $\lvert \pi \rvert \leq 2$.   Partitioning vertices of $G$ into two  nonempty sets  yields an $ic$-partition of $G$. Hence, $IC(G)=2$. Conversely, suppose that $IC(G)=2$. Let $\pi =\lbrace V_1,V_2 \rbrace$ be an $IC(G)$-partition. If both $V_1$ and $V_2$ are singleton dominating sets, then $G\simeq K_2$. Hence, we may assume that at least one of them (say $V_1$) is not a singleton dominating set. It follows that $V_2$ is not a singleton dominating set either, for otherwise, $G$ is a star, and so by   Observation \ref{obs-star}, we  have $IC(G)=3$. Hence, $V_1$ and $V_2$ are $ic$-partners, and so $V=V_1 \cup V_2$ is an independent set. Hence, $G\simeq \overline{K}_n$, for some $n\geq 2$.
\end{proof}
\begin{definition}
	Let  $\mathcal{B}_1$ represent the family of bipartite graphs $H$ with partite sets $H_1$ and $H_2$ such that $\lvert H_1\rvert \geq 2$,  $\lvert H_2 \rvert \geq 2$, $\delta(H)\geq 1$ and $id(H)=2$.
	
\end{definition}
\begin{definition}
	For $n\geq 1$, let  $\mathcal{B}_2$ represent the family of graphs $H \cup \overline{K}_m$, where $H$ is a bipartite graph with $\delta(H)\geq 1$ and $id(H)=2$.
\end{definition}
\begin{definition}
	For $n\geq 1$, let  $\mathcal{B}_3$ represent the family of  graphs $H \cup \overline{K}_m$, where $H$ is a $3$-partite graph with $\delta(H)\geq 1$ and $id(H)=3$.
\end{definition} 
\begin{proposition}\label{prop3}
	Let $G$ be a graph of order $n$. Then $IC(G)=3$ if and only if  $G\in \lbrace K_3, K_{1,n-1}\rbrace \cup \mathcal{B}_2$.
\end{proposition}
\begin{proof}
	Observations \ref{obs-comp} and \ref{obs-star} imply that $IC(K_3)=3$ and that $IC(K_{1,n-1})=3$, respectively. Now let $G \in \mathcal{B}_2$.
	Let $I$ be the set of isolates vertices of $G$, and let  $\lbrace H_1,H_2\rbrace$ be a partition of $G - I$ into its partite sets. We observe that the partition $\lbrace I,H_1,H_2\rbrace$ is an $ic$-partition of $G$, so $IC(G)\geq 3$. Now we show that $IC(G)=3$. Suppose, to the contrary, that $IC(G)\geq 4$. Let $\pi=\lbrace V_1,V_2,\dots ,V_k\rbrace $ be an $IC(G)$-partition. By Lemma \ref{lemiso}, we have $I \in \lbrace V_1,V_2,\dots ,V_k \rbrace$. Assume, without loss of generality, that  $I=V_1$. Now for each $2\leq i\leq k$, $V_i$  forms an independent coalition with $V_1$, and so $V_i$  dominates $H$. Hence, the partition $\lbrace V_2,\dots ,V_k\rbrace$ is an idomatic partition of $H$, which contradicts the assumption. Hence, $IC(G)=3$. Conversely,
	let $G$ be a graph with $IC(G)=3$, and let $\pi =\lbrace V_1,V_2,V_3 \rbrace$ be an $IC(G)$-partition. We consider four cases depending on the number of full vertices of $G$.
	
	\textbf{Case 1.} $G$ has three full vertices. In this case, the sets $V_1$, $V_2$ and $V_3$ are all singleton dominating sets, so $G\simeq K_3$. 
	
	\textbf{Case 2.} $G$ has two full vertices. Note that this case never occurs. 
	
	\textbf{Case 3.} $G$ has one full vertex. Let $v_1$ be the full vertex of $G$. Lemma \ref{lemfull} implies that $IC(G - v_1)=2$. Thus, by Proposition \ref{prop2}, either $G - v_1 \simeq K_2$, implying that $G\simeq K_3$, or $G - v_1 \simeq \overline{K}_n$, for some $n\geq 2$, which implies that $G\simeq K_{1,n-1}$, for some $n\geq 3$. 
	
	\textbf{Case 4.} $G$ has no full vertex. Let $I$ be the set of isolated vertices of $G$. First we note that  $V_1$, $V_2$ and $V_3$ are not pairwise $ic$-partners, for otherwise, we have $G \simeq \overline{K}_n$, and so by Proposition \ref{prop2}, we have $IC(G)=2$, a contradiction. Hence, $\pi$ contains a set (say $V_1$) that forms an independent coalition with $V_2$ and $V_3$, while $V_2$ and $V_3$ are not $ic$-partners. Therefore, each vertex in $V_1$ is an isolated vertex, so it follows from  Lemma \ref{lemiso} that  $I=V_1$. Further, the sets $V_2$ and $V_3$ are independent dominating sets of $G[V_2 \cup V_3]$, implying that $id (G[V_2 \cup V_3])\geq 2$. It remains to show that  $id (G[V_2 \cup V_3])=2$. Suppose, to the contrary, that $id (G[V_2 \cup V_3])\geq 3$. Let $\pi^\prime =\lbrace U_1,U_2,\dots ,U_k \rbrace$, ($k\geq 3$), be an idomatic partition of $G[V_2 \cup V_3]$. Then the partition $\pi^{\prime\prime} =\lbrace U_1,U_2,\dots ,U_k,V_1\rbrace$ is clearly an $ic$-partition of $G$, implying that $IC(G)\geq 4$, a contradiction. Hence, $G\in \mathcal{B}_2$.
\end{proof}

\begin{proposition}\label{prop4}
	Let $G$ be a graph. If $IC(G)=4$, then $G \in \lbrace K_4,K_2 + \overline{K}_n, K_1 + B \rbrace \cup \mathcal{B}_1 \cup \mathcal{B}_3$, where $n\geq 2$ and $B \in \mathcal{B}_2$. 
\end{proposition}
\begin{proof}
	Let $\pi=\lbrace V_1,V_2,V_3,V_4\rbrace$ be an $ic$-partition of $G$. We consider two cases.

	\textbf{Case 1.} $G$ has a full vertex. Let  $v_1$ be a full vertex of $G$. Lemma \ref{lemfull} implies that $IC(G - v_1)=3$. Thus, by Proposition \ref{prop3}, we have $G - v_1 \simeq K_3$, implying  that $G\simeq K_4$, or $G - v_1 \simeq K_{1,n}$, for some $n\geq 2$, implying that $G\simeq K_2 + \overline{K}_n$, for some $n\geq 2$, or $G -  v_1 \in \mathcal{B}_2$, which implies that   $G\simeq K_1 + B$, where $B\in \mathcal{B}_2$. 
	
	\textbf{Case 2.} $G$ has no full vertex. First  assume that $G$ contains a nonempty set $I$ of isolated vertices. Then, by Lemma \ref{lemiso}, we have $I \in \pi$. Without loss of generality, assume that $I=V_4$. Now for each $1\leq i\leq 3$, $V_i$ must form an independent coalition with $I$. Thus, $U=G[V_1 \cup V_2 \cup V_3]$ is a 3-partite graph with $id (U)\geq 3$. Since $IC(G)=4$, the case $id(U) >3$ is impossible. Hence, $id (U)=3$, and so $G\in \mathcal{B}_3$. Now assume that $G$ contains no isolated vertex. Since $G$ has neither full vertices nor isolated vertices, each set of $\pi$  has either one or two $ic$-partners. If there is a set of $\pi$, (say $V_1$)  having one $ic$-partner, (say $V_2$), then  it follows that $V_3$ and $V_4$ are  $ic$-partners, and so $G$ is a bipartitie graph with partite sets $V_1\cup V_2$ and $V_3 \cup V_4$. Otherwise, assume, without loss of generality, that $V_2$ and $V_3$ are $ic$-partners of $V_1$. It follows that $V_4$ has  an $ic$-partner in $\lbrace V_2,V_3\rbrace$. By symmetry, we may assume that $V_4$ and $V_3$ are $ic$-partners. Then $G$ is again a bipartitie graph with partite sets $V_1\cup V_2$ and $V_3 \cup V_4$.  Now using Theorem \ref{the-doma}, we have $id(G)=2$, and so  $G\in \mathcal{B}_1$.
\end{proof}

\section{Graphs with large independent coalition number}

Our main goal in this section is to investigate structure of graphs $G$ of order $n$  with $IC(G)=n$, under specified conditions. In addition, we will characterize all trees $T$ of order $n$ with $IC(T)=n-1$. 
Let us begin with an observation that characterizes all disconnected graphs $G$ of order $n$ with $IC(G)=n$. 

\begin{observation} \label{dis-n}
	Let $G$ be a disconnected graph of order $n$. Then  $IC(G)=n$ if  and only if $G\simeq K_s \cup K_r$, for some  $s\geq 1$, and  $r\geq 1$.
\end{observation}
Now we introduce two sufficient conditions for a  graph $G$ of order $n$ to have independent coalition number $n$.  
\begin{observation}
	If $G$ is a  graph of order $n$  with  $\alpha (G)=2$, then  $IC(G)=n$.
\end{observation}
\begin{proof}
	Let $G$ be a graph of order $n$ such that $\alpha (G)=2$. Consider the singleton partition $\pi_1$ of $G$. Note that for any two non-adjacent vertices $v$ and $u$ in $V(G)$, the sets $\lbrace v\rbrace$ and $\lbrace u\rbrace$ in $\pi_1$, are $ic$-partners. Hence, $\pi_1$ is an $ic$-partition of $G$, and so $IC(G)=n$.
	
\end{proof}
\begin{observation}
	Let $G$ be a  graph of order $n$. If $G$ admits a partition of its vertices into two maximal cliques, then  $IC(G)=n$.
\end{observation}
\subsection{Graphs $G$ with $\delta(G)=1$ and $IC(G)=n$}
In this subsection, we characterize graphs $G$ with $\delta (G)=1$ and $IC(G)=n$. We need the following definition.
\begin{definition}
	Let $G$ be a graph of order $n$, ($n\geq 3$), and let $\delta(G)=1$. Furthermore, let $x$ be a pendant vertex of $G$, and let $y$ be the support vertex of $x$. Then $G \in \mathcal{F}$ if  and only if $V(G) \setminus \lbrace x,y\rbrace$ induces clique.
\end{definition}
\begin{theorem} \label{l5}
	Let $G$ be a  graph of order $n$ with  $\delta(G)=1$.  Then $IC(G)=n$ if  and only if either $G\simeq K_2$, or $G \in \mathcal{F}$.
\end{theorem}

\begin{proof}
	Obviously, $IC(K_2)=2$. Now assume that $G \in \mathcal{F}$. Let $x$ be a pendant vertex of $G$, and let $y$ be the support vertex of $x$. Further, let $U=V(G) \setminus \lbrace x,y \rbrace$. Note that $U$ contains no full vertex. If $y$ is a full vertex, then $G $ is obtained from the complete graph $K_{n-1}$, where one of its vertices is adjacent to a leaf. In this case,  we clearly have $IC(G)=n$. Thus, we may assume that $y$ is not a full vertex, that is, there is a vertex $u \in U$ such that $u$ is not adjacent to $y$. Then it is easy to verify that the sets $\lbrace y\rbrace$ and $\lbrace u\rbrace$ are $ic$-partners, and that  each vertex in  $U \setminus \lbrace u\rbrace$ forms an independent coalition with $\lbrace x\rbrace$. Therefore, $IC(G)=n$.
	Conversely, suppose that $G$ is a graph with $\delta(G)=1$ and $IC(G)=n$. Let $x$ be a leaf of $G$, and let $y$ be the support vertex of $x$. Consider the singleton partition $\pi_1$ of $G$. Note that  each set in $\pi_1 \backslash \lbrace \lbrace x\rbrace ,\lbrace y\rbrace \rbrace$ must be an $ic$-partner of $\lbrace x\rbrace$ or $\lbrace y\rbrace$, to dominate $x$. Let $A=N(y)\setminus \lbrace x\rbrace$, and $B=V(G)\setminus (\lbrace x,y\rbrace \cup A)$. We consider four cases.
	
	\textbf{Case 1.} $A=\emptyset$ and $B=\emptyset$. In this case, we have $G\simeq K_2$. 
	
	\textbf{Case 2.} $A=\emptyset$ and $B\neq \emptyset$. By Observation \ref{dis-n}, we have $G \simeq K_2 \cup K_r$, for some $r\geq 1$. Thus, $G \in \mathcal{F}$.
	
	\textbf{Case 3.} $A \neq \emptyset$ and $B= \emptyset$. For each $v\in A$, the set $\lbrace v\rbrace$ cannot be an $ic$-partner of $\lbrace y\rbrace$, so it must be an $ic$-partner of $\lbrace x\rbrace$. This implies that $A$ induces a clique. Hence, $G \in \mathcal{F}$.
	
	\textbf{Case 4.} $A \neq \emptyset$ and $B \neq \emptyset$.
	For each $v\in A$, the set $\lbrace v\rbrace$ cannot be an $ic$-partner of $\lbrace y\rbrace$, so it must be an $ic$-partner of $\lbrace x\rbrace$. This implies that $[A ,B]$ is full and that $A$ induces a clique. Now for each vertex $u \in B$, in order for the set $\lbrace u \rbrace$ to be an $ic$-partner of $\lbrace x\rbrace$ or $\lbrace y\rbrace$,  $u$ must be adjacent to all other vertices in $B$. Hence, $B$ induces a clique, and so $G \in \mathcal{F}$, which completes the proof.
\end{proof}

As an  immediate  result from Theorem \ref{l5} we have: 
\begin{corollary} \label{tree-n}
	Let $T$ be a tree  of order $n$. Then $IC(T)=n$ if and only if  $T \in \lbrace P_1,P_2,P_3,P_4\rbrace$.
\end{corollary}

\subsection{Triangle-free graphs $G$ with $IC(G)=n$}
In this subsection, we   characterize graphs $G$  of order $n$ with $g(G)=4$ and $IC(G)=n$. This will lead to characterization of all triangle-free graphs $G$  of order $n$ with $IC(G)=n$. We will make use the following lemmas.
\begin{lemma} \label{girth7}
	Let $G$ be a triangle-free graph of order $n$ with $IC(G)=n$. Then $g(G)\leq 6$.
\end{lemma}
\begin{proof}
	Let $G$ be a graph of order $n$ with $IC(G)=n$, and suppose, to the contrary, that $g(G)\geq 7$. Let  $C \subseteq G$ be a cycle of order $g(G)$. Consider an arbitrary vertex $v\in V(C)$.  Note that $\gamma_i (C)\geq 3$, and so  $\lbrace v\rbrace$ is not an $ic$-partner of any set $\lbrace u\rbrace \subset V(C)$. Therefore, it must be an $ic$-partner of a set   $\lbrace u\rbrace \subseteq V(G) \setminus V(C)$. It follows that, $\lbrace u\rbrace$  dominates $V(C) \setminus N_c [v]$, which implies that $G$ contains triangles, a contradiction.
\end{proof}

\begin{lemma} \label{lem-girth6}
	Let $G$ be a  graph of order $n$ with $g(G)=6$. Then $IC(G)=n$ if and only if $G \simeq C_6$.
\end{lemma}

\begin{proof}
	Let $G$ be a graph of order $n$ with $g(G)=6$. If $G \simeq C_6$, then by Theorem \ref{th-cycle}, we have $IC(G)=6$. Conversely, assume that $IC(G)=n$. Let  $C \subseteq G$ be a cycle of order $6$, and suppose, to the contrary, that $V(G) \setminus V(C) \neq \emptyset$. Consider an arbitrary vertex $v\in V(G) \setminus V(C)$. If $\lbrace v\rbrace$ is an $ic$-partner of a set $\lbrace u\rbrace \subset V(C)$, then $\lbrace v\rbrace$ must dominate $V(C) \setminus N_c [u]$, which implies that $G$ contains triangles, a contradiction. Otherwise,
	$\lbrace v\rbrace$ must be an $ic$-partner of a set $\lbrace u\rbrace \subset V(G) \setminus V(C)$. Now since $\lbrace u,v\rbrace$ dominates $C$, it follows that $G$ contains triangles, or induces cycles of order 4, a contradiction.
	
\end{proof}
Our next result can be established almost the same way as  Lemma \ref{lem-girth6},  so we state it without proof.

\begin{lemma} \label{lem-girth5}
	Let $G$ be a   graph of order $n$ with $g(G)=5$. Then $IC(G)=n$ if and only if $G \simeq C_5$.
\end{lemma}

In order to characterize graphs $G$  of order $n$ with $IC(G)=n$ and $g(G)=4$, we need the following definitions.
\begin{definition}
	Let $\mathcal{K}_0$ represent a bipartite graph  with partite sets $H_1=\lbrace v_1,v_2,v_3,v_4\rbrace$ and $H_2=\lbrace u_1,u_2,u_3,u_4\rbrace$ such that for each $1\leq i\leq 4$, $v_i$ is adjacent to all vertices in $H_2$, except $u_i$. (see Figure 2).
\end{definition}

\begin{definition}
	Let $\mathcal{K}$ represent a family of $4$-partite graphs with partite sets $H_1 =\lbrace v_1,v_2,v_3,v_4\rbrace$, $H_2=\lbrace u_1,u_2,u_3,u_4\rbrace$, $H_3= \lbrace n_1,n_2,\dots ,n_k\rbrace$ and $H_4=\lbrace m_1,m_2,\dots ,m_k\rbrace$, for  $k\geq 1$, with the following properties:
	\begin{itemize}
		\item
		$[H_1,H_3]$ is full and $[H_2,H_4]$ is full,
		\item
		$[H_1,H_4]$ is empty and $[H_2,H_3]$ is empty,
		\item
		For each $1\leq i\leq 4$, $v_i$ is adjacent to all vertices in $H_2$, except $u_i$,
		\item
		For each $1\leq i\leq k$, $n_i$ is adjacent to all vertices in $H_4$, except $m_i$. 
		
	\end{itemize}
	Figure 3 illustrates such a graph for $k=3$.
\end{definition}
\begin{figure}[!htbp]
	\centering
	\begin{tikzpicture}[scale=.3, transform shape]
		\node [draw, shape=circle,fill=black] (v1) at  (0,4) {};
		\node [draw, shape=circle,fill=black] (v2) at  (4,4) {};
		\node [draw, shape=circle,fill=black] (v3) at  (8,4) {};
		\node [draw, shape=circle,fill=black] (v4) at  (12,4) {};
		\node [draw, shape=circle,fill=black] (u1) at  (0,0) {};
		\node [draw, shape=circle,fill=black] (u2) at  (4,0) {};
		\node [draw, shape=circle,fill=black] (u3) at  (8,0) {};
		\node [draw, shape=circle,fill=black] (u4) at  (12,0) {};
		\node [scale=3] at (0,4.7) {$v_1$};
		\node [scale=3] at (4,4.7) {$v_2$};
		\node [scale=3] at (8,4.7) {$v_3$};
		\node [scale=3] at (12,4.7) {$v_4$};
		\node [scale=3] at (0,-0.7) {$u_1$};
		\node [scale=3] at (4,-0.7) {$u_2$};
		\node [scale=3] at (8,-0.7) {$u_3$};
		\node [scale=3] at (12,-0.7) {$u_4$};
		\draw(v1)--(u3);
		\draw(v1)--(u4);
		\draw(v1)--(u2);
		\draw(v2)--(u1);
		\draw(v2)--(u3);
		\draw(v2)--(u4);
		\draw(v3)--(u1);
		\draw(v3)--(u2);	
		\draw(v3)--(u4);
		\draw(v4)--(u1);
		\draw(v4)--(u2);
		\draw(v4)--(u3);
		
	\end{tikzpicture}
	\caption{The graph $\mathcal{K}_0$}\label{pic2}
\end{figure}
\begin{figure}[!htbp]
	\centering
	\begin{tikzpicture}[scale=.3, transform shape]
		\node [draw, shape=circle,fill=black] (v1) at  (0,0) {};
		\node [draw, shape=circle,fill=black] (v2) at  (2,2) {};
		\node [draw, shape=circle,fill=black] (v3) at  (4,4) {};
		\node [draw, shape=circle,fill=black] (v4) at  (6,6) {};
		\node [draw, shape=circle,fill=black] (u4) at  (15,6) {};
		\node [draw, shape=circle,fill=black] (u3) at  (17,4) {};
		\node [draw, shape=circle,fill=black] (u2) at  (19,2) {};
		\node [draw, shape=circle,fill=black] (u1) at  (21,0) {};
		\node [draw, shape=circle,fill=black] (n1) at  (2,-7) {};
		\node [draw, shape=circle,fill=black] (n2) at  (4,-9) {};
		\node [draw, shape=circle,fill=black] (n3) at  (6,-11) {};
		\node [draw, shape=circle,fill=black] (m2) at  (17,-9) {};
		\node [draw, shape=circle,fill=black] (m1) at  (19,-7) {};
		\node [draw, shape=circle,fill=black] (m3) at  (15,-11) {};
		\node [scale=3] at (0,0.7) {$v_1$};
		\node [scale=3] at (2,2.7) {$v_2$};
		\node [scale=3] at (4,4.7) {$v_3$};
		\node [scale=3] at (6,6.7) {$v_4$};
		\node [scale=3] at (15,6.7) {$u_4$};
		\node [scale=3] at (17,4.7) {$u_3$};
		\node [scale=3] at (19,2.7) {$u_2$};
		\node [scale=3] at (21,0.7) {$u_1$};
		\node [scale=3] at (2,-7.7) {$n_1$};
		\node [scale=3] at (4,-9.7) {$n_2$};
		\node [scale=3] at (6,-11.7) {$n_3$};
		\node [scale=3] at (15,-11.7) {$m_3$};
		\node [scale=3] at (17,-9.7) {$m_2$};
		\node [scale=3] at (19,-7.7) {$m_1$};
		\draw(v1)--(u3);
		\draw(v1)--(u4);
		\draw(v1)--(u2);
		\draw(v2)--(u1);
		\draw(v2)--(u3);
		\draw(v2)--(u4);
		\draw(v3)--(u1);
		\draw(v3)--(u2);	
		\draw(v3)--(u4);
		\draw(v4)--(u1);
		\draw(v4)--(u2);
		\draw(v4)--(u3);
		\draw(n1)--(m2);	
		\draw(n2)--(m1);
		\draw(n1)--(m3);	
		\draw(n2)--(m3);
		\draw(n3)--(m1);	
		\draw(n3)--(m2);
		
		\draw(v1)--(u3);
		\draw(n1)--(v1);
		\draw(n1)--(v2);
		\draw(n1)--(v3);
		\draw(n1)--(v4);
		\draw(n2)--(v1);
		\draw(n2)--(v2);
		\draw(n2)--(v3);	
		\draw(n2)--(v4);
		\draw(n3)--(v1);
		\draw(n3)--(v2);
		\draw(n3)--(v3);	
		\draw(n3)--(v4);
		
		\draw(m1)--(u1);
		\draw(m1)--(u2);
		\draw(m1)--(u3);
		\draw(m1)--(u4);	
		\draw(m2)--(u1);
		\draw(m2)--(u2);
		\draw(m2)--(u3);
		\draw(m2)--(u4);
		\draw(m3)--(u1);
		\draw(m3)--(u2);
		\draw(m3)--(u3);
		\draw(m3)--(u4);
	\end{tikzpicture}
	\caption{The graph in $\mathcal{K}$ for $k=3$}\label{pic3}
\end{figure}

\begin{theorem} \label{the-girth4}
	Let $G$ be a graph  of order $n$ with $g(G)=4$. Then $IC(G)=n$ if and only if $G\in  \{C_4, \mathcal{K}_0\} \cup \mathcal{K}$.
\end{theorem}
\begin{proof}
	It is easy to check that $IC(C_4)=4$ and that $IC(\mathcal{K}_0)=8$. Now
	let $G\in \mathcal{K}$. We observe that for each $1\leq i\leq 4$, $\lbrace v_i\rbrace$ and $\lbrace u_i\rbrace$ are $ic$-partners, and that for each $1\leq i\leq k$, $\lbrace n_i\rbrace$ and $\lbrace m_i\rbrace$ are $ic$-partners. Thus, $IC(G)=n$. Conversely, let $G$ be a graph  of order $n$ with $g(G)=4$ and $IC(G)=n$, and let $C$ be a cycle of $G$ of order 4 with  $V(C)=\lbrace x,y,z,t\rbrace$ and $E(C)=\lbrace xy,yz,zt,tx\rbrace$. If $G=C$, then the desired result follows. Hence, we assume that $G\neq C$. Since $x$ is adjacent to $y$ and $t$, neither $\lbrace y\rbrace$ nor  $\lbrace t\rbrace$ is an $ic$-partner of $\lbrace x\rbrace$. Now consider two cases. 
	
	\textbf{Case 1.} $\lbrace x\rbrace$ and $\lbrace z\rbrace$ are  $ic$-partners. In this case, $G$ is dominated by $\lbrace x,z\rbrace$. Let $A=N(x) \setminus \lbrace y,t\rbrace$ and $B=N(z) \setminus \lbrace y,t\rbrace$. If $A \neq \emptyset$, (say $v \in A$), then it is not difficult to check that $\lbrace v\rbrace$ has no $ic$-partner. Thus, $A=\emptyset$, and so by symmetry, we have $B =\emptyset$. Hence, $G \simeq C_4$. 
	
	\textbf{Case 2.} $\lbrace x\rbrace$ and $\lbrace z\rbrace$ are not $ic$-partners.
	Let $\lbrace e\rbrace$ be an $ic$-partner of $\lbrace x\rbrace$.  Since $\lbrace x,e\rbrace$ dominates $G$ and  $z$ is not adjacent to  $x$ , it must be adjacent to $e$.  Let $A=N(x) \setminus \lbrace y,t\rbrace$ and $B=N(e) \setminus \lbrace z\rbrace$. It is not difficult to verify that $A \cap B = \emptyset$. Now if $A  = \emptyset$, then $\lbrace z\rbrace$ cannot form an independent coalition with any other set, so $A \neq \emptyset$. Let $\lbrace f\rbrace \subseteq A$ be an $ic$-partner of $\lbrace z\rbrace$. We note that if a set $\lbrace g\rbrace$ forms an independent coalition with $\lbrace y \rbrace$, then $g\in B$. Further, if  a set $\lbrace h\rbrace$ forms an independent coalition with $\lbrace t \rbrace$ then $h\in B$. Let $\lbrace g\rbrace$ and $\lbrace h\rbrace$ be  $ic$-partners of $\lbrace y\rbrace$ and $\lbrace t\rbrace$, respectively. Observe that $\lbrace g\rbrace \neq \lbrace h\rbrace$. Now  let $A^\prime =A\setminus \lbrace f\rbrace$ and $B^\prime =B\setminus \lbrace g,h\rbrace$. There exist the following subcases. 
	
	\textbf{Subcase 2.1.} $A^\prime =\emptyset$ and $B^\prime =\emptyset$. In this case, we have $G \simeq \mathcal{K}_0$. 
	
	\textbf{Subcase 2.2.} $A^\prime = \emptyset$ and $B^\prime \neq \emptyset$.
	Let $v\in B^\prime$. One can verify that $\lbrace v\rbrace$ cannot form an independent coalition with any other set. Thus, this case is impossible. 
	
	\textbf{Subcase 2.3.} $A^\prime \neq \emptyset$ and $B^\prime = \emptyset$. Let $v\in A^\prime$. One can verify that $\lbrace v\rbrace$ cannot form an independent coalition with any other set. Thus, this case is impossible. 
	
	\textbf{Subcase 2.4.} $A^\prime \neq \emptyset$ and $B^\prime \neq \emptyset$. Let $v\in A^\prime$. If a set $\lbrace u\rbrace$ forms an independent coalition with $\lbrace v\rbrace$, then $u\in B^\prime$. Furthermore, for each vertex  $ u \in B^\prime$, $\lbrace u\rbrace$ cannot form an independent coalition with more than one sets $\lbrace v\rbrace \subseteq A^\prime$. Thus, $\lvert A^\prime \rvert \leq \lvert B^\prime \rvert$.  Using a similar argument, we deduce that  $\lvert B^\prime \rvert \leq \lvert A^\prime \rvert$, and so $\lvert A^\prime \rvert = \lvert B^\prime \rvert$. Consequently, the following statements hold in the graph $G$:
	\begin{itemize}
		\item
		$G[\lbrace x,y,z,t,e,f,g,h\rbrace]$ is a bipartite graph with partite sets $V_1=\{ x,z,g,h\}$ and $V_2=\{ y,t,e,f\}$, which is isomorphic to $\mathcal{K}_0$,
		\item 
		$[V_1,A^\prime]$ is full and $[V_2,B^\prime]$ is full,
		\item
		$[V_1,B^\prime]$ is empty and $[V_2,A^\prime]$ is empty,
		\item
		$G[A^\prime \cup B^\prime]$ is a bipartite graph with partite sets $A^\prime$ and $B^\prime$ such that $deg_{G[A^\prime \cup B^\prime]}(v)=\lvert A^\prime \rvert -1=\lvert B^\prime \rvert -1$, for each $v\in A^\prime \cup B^\prime$.
	\end{itemize}
	Hence, $G\in \mathcal{K}$ and the proof is complete.	
	
\end{proof}
Using Observation \ref{dis-n},
Corollary \ref{tree-n}, Lemmas \ref{girth7} ,\ref{lem-girth6} and \ref{lem-girth5} , and Theorem \ref{the-girth4}, we infer the following result.

\begin{corollary}
	Let $G$ be a triangle-free graph of order $n$. Then $IC(G)=n$ if and only if 
	$G\in \lbrace C_4,C_5,C_6,P_1,P_2,P_3,P_4,\overline{K}_2,K_1 \cup K_2,K_2 \cup K_2,\mathcal{K}_0\rbrace \cup  \mathcal{K} $.
\end{corollary}	
	
	\subsection{Trees $T$ with $IC(T)=n-1$}
	The following theorem characterizes all trees $T$ of order $n$ with $IC(T)=n-1$.
	\begin{theorem}\label{the-t=n-1}
		Let $T$ be a tree of order $n$. Then $IC(T)=n-1$ if and only if $T\in \lbrace P_5,P_6,S_{1,2},K_{1,3}\rbrace$.
	\end{theorem}
	\begin{proof}
		By Theorem \ref{the-path}, we have $IC(P_5)=4$ and $IC(P_6)=5$. Further, by Observation \ref{obs-star}, we have $IC(K_{1,3})=3$ and by Observation \ref{obs-dstar}, we have $IC(S_{1,2})=4$. Conversely,
		let $T$ be a tree of order $n$ with $IC(T)=n-1$, where  $x$ is a leaf , and  $y$ is the support vertex of $x$. Define $A=N(y) \setminus \{x\}$ and $B=V(G) \setminus  ( \lbrace x,y \rbrace \cup A )$. Further, let $\pi$ be an $IC(T)$-partition. Note that $\pi$ contains a set of cardinality $2$ (say $V_1=\lbrace u,v\rbrace$)  and $n-2$ singleton sets. Since $x$ and $y$ are adjacent, we have $V_1 \neq \lbrace x,y\rbrace$. Note as well that any set in $\pi$ must be an $ic$-partner of  the set containing $x$, or the set containing $y$, to dominate $x$. We consider two cases.
		
		\textbf{Case 1.}  $B= \emptyset$. If   $A=\emptyset$, then we have $T\simeq K_2$, and so  $IC(T)=2 \neq n-1$. Hence, $A \neq \emptyset$, an so  $T\simeq K_{1,n-1}$, for some $n\geq 3$. Now by Lemma \ref{l3}, we have $IC(T)=3$. Hence, $T\simeq K_{1,3}$.
		
		\textbf{Case 2.}  $B\neq \emptyset$. Since $T$ is connected, we have $A\neq \emptyset$. We divide this case into some subcases.
		
		\textbf{Subcase 2.1.} $u\in A$ and $v\in B$.
		We first show that $\lvert A\rvert =1$. Suppose, to the contrary, that $\lvert A\rvert \geq 2$. Then  there is a vertex $z\in A$ such that $z\neq u$. Since $z$ and $y$ are adjacent,  $\lbrace z\rbrace$ cannot be an $ic$-partner of  $\lbrace y\rbrace$, so it must be an $ic$-partner of  $\lbrace x\rbrace$. Since $\lbrace x\rbrace$ does not dominate $u$,  $u$ must be adjacent to $z$, which is a contradiction, since $y$, $z$ and $u$ induce a triangle.
		Now  $\lbrace y\rbrace$ cannot be an $ic$-partner of $\lbrace x\rbrace$ or $\lbrace u,v\rbrace$, so it must have an $ic$-partner in $B$. This implies that $\lvert B\rvert \geq 2$. Let $\lbrace t\rbrace \subset B$ be an $ic$-partner of $\lbrace y\rbrace$.
		we show that $B \setminus \lbrace v,t\rbrace = \emptyset$. Suppose that $B \setminus \lbrace v,t\rbrace \neq \emptyset$. Let $z\in B \setminus \lbrace v,t\rbrace$. Note that $v$ and $t$ are adjacent. Now  $\lbrace z\rbrace$ must be an $ic$-partner of $\lbrace x\rbrace$ or $\lbrace y \rbrace$, so $z$ must be adjacent to $t$ and $v$, which is a contradiction, since $z$,$t$ and $v$ induce a triangle.
		Hence, $B=\lbrace v,t\rbrace$ and so $T \simeq P_5$.
		
		\textbf{Subcase 2.2.} $\{u,v\}\subseteq B$. An  argument similar to the one presented above implies that $\lvert A\rvert =1$.
		Now we show that $B\setminus \lbrace u,v\rbrace =\emptyset$. Suppose that $B\setminus \lbrace u,v\rbrace \neq \emptyset$. Let $z \in B\setminus \lbrace u,v\rbrace$, and let $A=\lbrace t\rbrace$. Since $t$ and $y$ are adjacent,  $\lbrace t \rbrace$ cannot be an $ic$-partner of $\lbrace y \rbrace$, so it must be an $ic$-partner of   $\lbrace x\rbrace$. Thus, $t$ must be adjacent to $u$, $v$ and $z$. Now $\lbrace z\rbrace$ must be  an $ic$-partner of  $\lbrace x\rbrace$ or  $\lbrace y\rbrace$, so $z$ must be adjacent to $u$ and $v$, which is a contradiction, since $z$, $u$ and $t$ induce a triangle. Hence, $T\simeq S_{1,2}$.
		
		\textbf{Subcase 2.3.} $\lbrace u,v\rbrace \subseteq A$. We first show that $A \setminus \lbrace u,v\rbrace =\emptyset$. Suppose that $A\setminus \lbrace u,v\rbrace \neq \emptyset$. Let $z \in A\setminus \lbrace u,v\rbrace$.
		Since $z$ is adjacent to $y$,  $\lbrace z\rbrace$ must be an $ic$-partner of $\lbrace x\rbrace$, so $z$ must be adjacent to $u$ and $v$, which is a contradiction, since $z$,$u$ and $y$ induce a triangle. Now we show that $\lvert B\rvert=1$. Suppose that $\lvert B\rvert \neq 1$ . First assume $\lvert B\rvert \geq 3$. Let $z,t,w\in B$. Now $z$,$t$ and $w$ induce a triangle, since the sets containing each of them, must be an $ic$-partner of $\lbrace x\rbrace$ or $\lbrace y\rbrace$, a contradiction. Now assume  $\lvert B\rvert = 2$. Let $B=\lbrace z,t\rbrace$. Each of the sets $\lbrace z\rbrace $ and $\lbrace t\rbrace$ must be an $ic$-partner of $\lbrace x\rbrace$ or $\lbrace y\rbrace$. Thus, $z$  must be adjacent to $t$. Now  $\lbrace u,v\rbrace$ must be an $ic$-partner of $\lbrace x\rbrace$, so $z$ and $t$ must be   dominated by $\lbrace u,v\rbrace$. Now the induced subgraph $T[\lbrace u,v,z,t\rbrace]$ contains at least one cycle, a contradiction.
		Hence, we have $T\simeq S_{1,2}$.

		\textbf{Subcase 2.4.} $u=y$ and $v\in B$.
		we first show that $\lvert A\rvert =1$. Suppose that $\lvert A\rvert \geq 2$. Let $z,t\in A$. Since $z$ and $t$ are adjacent to $y$,   $\lbrace z\rbrace$ and $\lbrace t\rbrace$ cannot be an $ic$-partner of $y$, so each of them  must be an $ic$-partner of $\lbrace x\rbrace$. Thus, $z$ must be adjacent to $t$, which is a contradiction, since $z$, $t$ and $y$ induce a triangle. Now we show that $\lvert B\rvert = 2$.  Suppose that $\lvert B\rvert \neq 2$ . If  $\lvert B\rvert = 1$, then $T \simeq P_4$, a contradiction. Otherwise, let $\lbrace v,z,t\rbrace \subseteq B$ and $A=\lbrace w\rbrace$. Now $\lbrace w\rbrace $ cannot be an $ic$-partner of  $\lbrace u,v\rbrace$, so it must be an $ic$-partner of  $\lbrace x\rbrace$. Thus, $w$ must be adjacent to $v$, $z$ and $t$. Now observe that $\lbrace u,v\rbrace$ must have an $ic$-partner in $B$. Assume, without loss of generality, that $\lbrace u,v\rbrace$ and $\lbrace z\rbrace$ are $ic$-partners. This implies that $t$ is adjacent to $z$ or $v$, which is impossible, since both cases lead to existence of an induced triangle.
		Hence, $T\simeq S_{1,2}$.

		\textbf{Subcase 2.5.} $u=x$ and $v\in A$.
		We first show that $\lvert B\rvert \leq 2$. Suppose that $\lvert B\rvert \geq 3$. Let $\lbrace z,t,w\rbrace \subseteq B$. Note that  $\lbrace y\rbrace$ must have an $ic$-partner in $B$. Assume, without loss of generality, that  $\lbrace y\rbrace$ and $\lbrace z\rbrace$ are $ic$-partners. It follows that $z$ is adjacent to $t$ and $w$. Now if  $\lbrace y\rbrace$ is an $ic$-partner of  $\lbrace t\rbrace$ or  $\lbrace w\rbrace$, then $t$ must be adjacent to $w$, which is impossible, since $z$,$t$ and $w$ induce a triangle. Hence, both $t$ and $w$ must be $ic$-partners of $\lbrace u,v\rbrace$, which implies that $t$ is adjacent to $w$. Now $z$,$t$ and $w$ induce a triangle, a contradiction. Now we show that $\lvert A\rvert \leq 2$. Suppose that $\lvert A\rvert \geq 3$. 
		Let $\lbrace z,t,v\rbrace \subseteq A$. The sets $\lbrace z\rbrace$ and $\lbrace t\rbrace$ must be $ic$-partners of $\lbrace u,v\rbrace$. This implies that $z$ is adjacent to $t$. Now $z$,$t$ and $y$ induce a triangle, a contradiction.
		Further, we observe that the case $\lvert A\rvert=\lvert B\rvert=2$ is impossible. Hence, either $\lvert A\rvert=2$ and $\lvert B\rvert =1$, which implies that $T\simeq S_{1,2}$, or $\lvert A\rvert =1$ and $\lvert B\rvert=2$, which implies that $T\simeq P_5$.
		
		\textbf{Subcase 2.6.} $u=x$ and $v\in B$.
		We first show that $\lvert A\rvert =1$. Suppose  that $\lvert A\rvert \geq 2$. Let $\lbrace z,t\rbrace \subseteq A$. Now each of the sets $\lbrace z\rbrace$ and $\lbrace t\rbrace$ must be an $ic$-partner of $\lbrace u,v\rbrace$. This implies that $z$ is adjacent to $t$, which is a contradiction, since $y$,$z$ and $t$ induce a triangle. Now we show that $\lvert B \rvert \leq 3$. Suppose  that $\lvert B\rvert \geq 4$. Let  $\lbrace v,t,w,h\rbrace \subseteq B$ and $A=\lbrace z\rbrace$. Note that  $\lbrace y\rbrace$ must have an $ic$-partner in $B$. Assume, without loss of generality, that $\lbrace y\rbrace$ and $\lbrace t\rbrace$ are $ic$-partners. It follows that $t$ is adjacent to $w$, $h$ and $v$. Now $\lbrace w\rbrace$ must be an $ic$-partner of $\lbrace y\rbrace$ or $\lbrace u,v\rbrace$. One can observe that both cases lead to contradiction.
		Hence, either $\lvert B \rvert = 2$, which implies that $T \simeq P_5$, or $\lvert B \rvert = 3$, which implies that $T\simeq P_6$.
	\end{proof}
	
	\section{Discussion and conclusions}
	In Proposition \ref{claim-1}, we introduced a family of graphs admitting no $ic$-partition. This result motivates the following problem:
	
	\textbf{Problem} 6.1. Characterize graphs admitting an  $ic$-partition.

	In Observations \ref{obs1} and \ref{obs2}, we presented the sharp inequalities $IC(G)\le C(G)$ and $IC(G)\geq \chi(G)$. This raises the following problems:
	
	\textbf{Problem} 6.2. Characterize graphs $G$ in which the equality  $IC(G)=C(G)$ holds.  	
	
	\textbf{Problem} 6.3. Characterize graphs $G$ in which the equality $IC(G)=\chi(G)$ holds.
	
	In Theorem \ref{the-t=n-1}, trees $T$ of order $n$ with $IC(T)=n-1$ have been characterized. This raises the following problem:
	
	\textbf{Problem} 6.4. Characterize graphs $G$ of order $n$ with $IC(G)=n-1$.


\begin{thebibliography}{99}
		
		\bibitem{ref1} A. Bu\v{g}ra \"{O}zer, E. Saygı and Z. Saygı, Domination type parameters of Pell graphs, Ars Mathematica Contemporanea, 23(1) (2023) 
		https://doi.org/10.26493/1855-3974.2637.f61.
		
		\bibitem{ref2} G.J. Chang, The domatic number problem, Discrete Math. 125 (1994) 115-122.
		
		\bibitem{ref3} A. Gorzkowska,  M.A. Henning,  M. Pil\'{s}niak and  E. Tumidajewicz, Paired domination stability in graphs, Ars Mathematica Contemporanea, 22(2) (2022) 
		https://doi.org/10.26493/1855-3974.2522.eb3.
		
		
		\bibitem{ref4} T.W. Haynes, J.T. Hedetniemi, S.T. Hedetniemi, A.A. McRae and R. Mohan, Coalition graphs, Commun. Comb. Optim., 8(2) (2023) 423-430,
		DOI: 10.22049/CCO.2022.27916.1394.
		
		\bibitem{ref5} T.W. Haynes, J.T. Hedetniemi, S.T. Hedetniemi, A.A. McRae, and R. Mohan,
		Coalition graphs of paths, cycles and trees, Discuss. Math. Graph Theory, (2020),
		DOI:https://doi.org/10.7151/dmgt.2416.
		
		\bibitem{ref6} T.W. Haynes, J.T. Hedetniemi, S.T. Hedetniemi, A.A. McRae and R. Mohan, In-
		troduction to coalitions in graphs, AKCE Int. J. Graphs Comb., 17 (2020) 653–659,
		https://doi.org/10.1080/09728600.2020.1832874.
		
		\bibitem{ref7}   T.W. Haynes, J.T. Hedetniemi, S.T. Hedetniemi,
		A.A. McRae, and R. Mohan, Self-coalition graphs, Opuscula Math., 43(2) (2023) 173-183,
		https://doi.org/10.7494/OpMath.2023.43.2.173.
		
		\bibitem{ref8} T.W. Haynes, J.T. Hedetniemi, S.T. Hedetniemi, A.A. McRae and R. Mohan, Upper bounds on
		the coalition number, Australas. J. Comb., 80(3)  (2021) 442-453.
		\bibitem{ref9} T.W. Haynes, J.T. Hedetniemi, S.T. Hedetniemi,
		A.A. McRae and N. Phillips, The upper domatic number of a graph,
		AKCE International Journal of Graphs and Combinatorics
		17(1) (2020) 139-148, https://doi.org/10.1016/j.akcej.2018.09.003.
		
		\bibitem{ref10} T.W. Haynes, S.T. Hedetniemi and P.J. Slater, Fundamentals of Domination in Graphs,  Marcel Dekker, New York, (1998).
		
		
		\bibitem{ref11} A.C. Mart\'{i}nez, A note on the $k$-tuple domination number of graphs, Ars Mathematica Contemporanea, 22(4) (2022) 
		https://doi.org/10.26493/1855-3974.2600.dcc.
		
		\bibitem{ref12} A.C. Mart\'{i}nez, A. Estrada-Moreno and J.A. Rodr\'{i}guez-Vel\'{a}zquez, From Italian domination in lexicographic
		product graphs to $w$-domination in graphs, Ars Mathematica Contemporanea, 22(1) (2022) 
		https://doi.org/10.26493/1855-3974.2318.fb9.
		
		\bibitem{ref13} D.B. West, Introduction to Graph Theory (Second Edition), Prentice Hall, USA, (2001).
	\end{thebibliography}
\end{document}